\documentclass[12pt]{article}
\usepackage[utf8]{inputenc}
\usepackage{lmodern}
\usepackage{graphicx, xcolor, mathrsfs}
\usepackage{amsfonts}
\usepackage{stmaryrd}
\usepackage{blkarray, bigstrut}
\usepackage{mathtools}
\usepackage[nice]{nicefrac}
\usepackage{amsmath,amsthm,amssymb}
\usepackage[mathlines]{lineno}
\usepackage{mathabx}
\usepackage{accents}

\usepackage{enumerate}
\usepackage[colorlinks=true,citecolor=black,linkcolor=black,urlcolor=blue]{hyperref}
\usepackage[top=.9in, bottom=.9in, left=.5in , right=.5in]{geometry}
\usepackage{caption}
\usepackage{xifthen} 
\usepackage{wrapfig}
\usepackage[numbers, square]{natbib}
\usepackage{changes}
\usepackage{comment}
\usepackage{float}
\usepackage{ifthen}
\mathtoolsset{showonlyrefs}
\usepackage{tikz}
\usepackage{caption}                       
\usetikzlibrary{arrows}
\usetikzlibrary{graphs,graphs.standard}
\usetikzlibrary{positioning,arrows.meta}
\usetikzlibrary{shapes.multipart}
\usetikzlibrary{arrows}
\usetikzlibrary{automata}
\definecolor{arrowblue}{RGB}{0,0,0}  % change colour of arrows in picture. I like black

\usepackage[splitrule]{footmisc} %% The splitrule option draws a full width rule above the continued part of the footnote as a visual cue to readers.
\usepackage{caption}
\newtheorem{theorem}{Theorem}[section]
\newtheorem{lemma}[theorem]{Lemma}

\newtheorem{proposition}[theorem]{Proposition}
\newtheorem{corollary}[theorem]{Corollary}
\newtheorem{assumption}{Assumption}[section]

\theoremstyle{definition}
\newtheorem{remark}{Remark}[section]
\newtheorem{example}[remark]{Example}

\newtheorem{definition}[assumption]{Definition}

\newcommand{\E}[2][]{\ensuremath{\mathbb{E}_{#1}\left[#2 \right]}}
\newcommand{\Prob}[2][]{\ensuremath{\mathbb{P}_{#1} \left(#2 \right)}}

\newcommand{\gen}[1]{\ensuremath{G_{#1}}}

\newcommand{\eps}{\varepsilon}

\newcommand{\Supe}{\mathcal{S}}

\DeclareMathOperator{\tpath}{path}

\DeclareMathOperator{\degg}{deg}
\DeclareMathOperator{\de}{d}
\newcommand{\outdeg}[1]{\ensuremath{\degg^{+}(#1)}}

\newboolean{parentmodel}
\setboolean{parentmodel}{false} %toggle first part

\title{On a sufficient condition for explosion in CMJ branching processes and applications to recursive trees}
\author{T. Iyer\footnote{Weierstrass Institute for Applied Analysis and Stochastics, Mohrenstrasse 39, 10117 Berlin, Germany.}}
\date{November 19, 2024}

\begin{document}
\maketitle
\abstract{We provide sufficient criteria for explosion in Crump-Mode-Jagers branching process, via the process producing an infinite path in finite time. As an application, we deduce a phase-transition in the infinite tree associated with a class of recursive tree models with fitness, showing that in one regime every node in the tree has infinite degree, whilst in another, the tree is locally finite, with a unique infinite path. The latter class encompasses many models studied in the literature, including the weighted random recursive tree, the preferential attachment tree with additive fitness, and the Bianconi-Barab\'asi model, or preferential attachment tree with multiplicative fitness.}
\section{Introduction}
In Crump-Mode-Jagers (CMJ) branching processes, an ancestral individual produces offspring according to a random collection of ``birth times'' on the non-negative real line. At each birth time, a new individual begins to produce offspring, according to a collection of points identically distributed to those associated with the ancestral individual, but now shifted by their birth time, and one is generally interested in features of the population as time varies. For example, if for some $t \in [0,\infty)$, there exists infinitely many individuals at time $t$, one says that explosion occurs. 

Suppose that $\xi(t)$ denotes the random number of individuals produced by the ancestral individual by time $t$ (with $\xi(0)$ denoting the number of individuals produced ``instantaneously''). A result of Komj\'athy~\cite[Theorem~3.1(b)]{Komjathy2016ExplosiveCB} states that, if for all $t > 0$ $\xi(t) < \infty$ almost surely, and for some $t_0 > 0$, we have $\E{\xi(t_0)} < 1$, explosion does not occur. The latter condition is equivalent to having $\E{\xi(0)} < 1$ and $\E{\xi(t_0)}<\infty$ for some $t_0 > 0$. In this short note, we provide a sufficient condition that guarantees the emergence of an infinite ancestral path of individuals in finite time in CMJ branching processes (see Theorem~\ref{theorem:vert-expln} and Corollary~\ref{cor:vert-expln}). This is thus also a sufficient condition for explosion. 

CMJ processes have been well studied in the so called \emph{Malthusian} regime (for example~\cite{kingman1975, nerman_81, olofsson-x-log-x}), however, far fewer results are available outside the Malthusian regime, including, in particular, the regime where explosion may occur. In this regard, Komj\'athy~\cite{Komjathy2016ExplosiveCB} has made a number of foundational contributions, under the assumption that $\xi(t) < \infty$ almost surely for all $t > 0$, (that is, the assumption that \emph{sideways explosion} does not occur). These include characterising the distribution of the explosion time in terms of the solution of a functional fixed point equation. Necessary and sufficient conditions for explosion in \emph{branching random walks} with plump offspring distributions have been provided in~\cite{amini-et-al}, where the authors showed that this property is equivalent, in this case, to a weaker condition called min-summability. This was extended to \emph{age dependent branching processes}, and \emph{processes with incubation} in~\cite[Section~5 and Section~7]{Komjathy2016ExplosiveCB}. In~\cite{bertoin-local-expl}, the authors provided sufficient criteria for \emph{local explosion} in related growth-fragmentation processes. Recently, when one \emph{allows} for the possibility of sideways explosion, sufficient criteria have been provided for explosion to occur via an infinite path, or by a single node having infinite degree in~\cite{inhom-sup-pref-attach}. 

\subsection{Random recursive trees with fitness} \label{sec:rec-fit}
One commonly used application of CMJ branching process is to evolving random tree models (an approach originating in~\cite{pittel}). 
In \emph{recursive trees with fitness}, nodes $n$, labelled by the natural numbers, arrive one at a time, and are assigned an i.i.d. random weight $W_{n}$. This node then is connected with an edge directed outwards\footnote{In different formulations of this model the edge is directed \emph{inwards} instead, and the fitness function is a function of the \emph{in-degree}.} from a randomly selected target node $v$. The node $v$ is selected at the $n$th time-step with probability proportional to a \emph{fitness} function $f(\outdeg{v},W_{v})$ which incorporates information about the out-degree of $v$ at the $n$th time-step. Particular instances of this model include 
inhomogeneous preferential attachment models, such as the Bianconi-Barab\'asi model (introduced in~\cite{bianconibarabasi2001}) and preferential attachment with additive fitness (introduced in~\cite{ergun}); and the weighted random recursive tree~\cite{wrt1,Sen21}. Inhomogeneous preferential attachment models are often used as models for the evolution of complex networks arising in applications. 

Interesting effects can occur in these recursive tree models, often as a result of the inhomogeneities. Suppose that $N_{k}(n)$ denotes the number of nodes of out-degree $k$ in the model at the $n$th time-step. Note that this implies that $\sum_{k=0}^{\infty} k N_{k}(n)/n = 1$, since there are $n$ edges in the tree at the $n$th time-step. Suppose also that $\lim_{n \to \infty} \frac{N_{k}(n)}{n}$ exists for all $k$, which we denote by $p_{k}$. It is conjectured~\cite{rec-trees-fit} that, if \[\sum_{j=1}^{\infty} \mathbb{E}\left[\prod_{i=0}^{j-1} \frac{f(i,W)}{f(i,W) + \lambda}\right] < 1,
\quad \text{for any $\lambda > 0$ such that the sum converges,}
\] 
then $\sum_{k=0}^{\infty} k p_{k} < 1$. This indicates that a positive fraction of `mass' of edges has been lost to a sub-linear number of nodes of `large degrees'. This effect, which has been proved in a number of particular instances, is known as \emph{condensation}  (see, for example,~\cite{borgs-chayes, dereich-mailler-morters}). Moreover, if the sum does not converge for \emph{any} $\lambda > 0$, an \emph{extreme condensation effect} is believed to occur, where $\sum_{k=0}^{\infty} k p_{k} = 0$. One of the goals of this article is to investigate this extreme condensation regime, in the scenario when $\sum_{i=1}^{\infty} \frac{1}{f(i,W)} = \infty$ almost surely. By drawing connections to related work in~\cite{inhom-sup-pref-attach}, we show that a phase transition occurs in this regime. In one phase every node in the infinite tree associated with the model has maximum degree possible (which may be infinite), whilst in another phase, this tree is locally finite, with a unique infinite path. One may interpret the latter phase as an extreme effect of competition from `newer' nodes, with so much mass of edges carried away by these nodes, that any individual in the infinite tree only has finite out-degree. 

\subsection{Overview}
Section~\ref{sec:preliminaries} deals with preliminaries: a formal description of the model (which generalises the classical CMJ framework) in Section~\ref{sec:notation}, and measure theoretic technicalities in Section~\ref{sec:measurability}. 
The main results appear in Section~\ref{sec:main}: in Theorem~\ref{theorem:vert-expln} and Corollary~\ref{cor:vert-expln} we provide a sufficient condition for explosion to occur via an infinite path appearing in finite time. These results use facts from~Section~\ref{sec:inherited} where we note that a number of properties of CMJ processes that are `inherited', including explosion, occur with probability $0$ or $1$, when one conditions on `survival' of the process. Section~\ref{sec:applications} deals with applications to the recursive tree with fitness model: Theorem~\ref{theorem:rec-trees-structure} provides a classification of the structure of the infinite tree $\mathcal{T}_{\infty}$ associated with this model, whilst Theorem~\ref{theorem:inf-path-trans} applies these results to a `linear fitness' regime, proving the aforementioned phase-transition. The linear fitness regime encompasses many existing models including the Bianconi-Barab\'asi model, preferential attachment model with additive fitness, and the weighted random recursive tree.

%General criteria for explosion have been provided in terms of the solution of a functional fixed point equation by Komj\'athy~\cite{Komjathy2016ExplosiveCB}, who also extended the necessary and sufficient criteria for explosion in branching random walks with plump offspring distribution in~\cite{amini-et-al}. We refer the reader to \cite{Komjathy2016ExplosiveCB} for a more comprehensive overview of the literature related to explosion in branching processes. In~\cite{bertoin-local-expl}, the authors provide sufficient criteria for \emph{local explosion} in closely related growth-fragmentation processes.
%is, if $\xi$ denotes the offpring distribution associated with the process, that for all $t \in [0, \infty)$, $\E{\xi([0, t))} = \infty$. In this short note, we derive sufficient criteria on $\xi$ for the converse to hold.We denote by \[N(t):= \left|u \in \mathcal{U}: \sigma_{u} \leq t \right|, \] the number of individuals in the process up to time $t$,
%In what follows, we denote by 
%\begin{equation}
%    f(x,t) := \Prob{\xi([0,t)) > x}.
%\end{equation}
%In this paper, we will consider family trees of explosive CMJ branching processes; which are explosive in the sense that they explode in finite time. {\color{red} Such models were introduced by Komj\'{a}thy in \cite{Komjthy2016ExplosiveCB}, in the setting where a single individual cannot produce infinitely many children in finite time; and thus, the author showed (among many other things) that explosion implies the existence of an infinite path.}
\section{Preliminaries} \label{sec:preliminaries}
\subsection{Notation and model preliminaries} \label{sec:notation}

In general, we will follow similar notation to that used in~\cite{inhom-sup-pref-attach}. First, we use $\mathbb{N} := \{1, 2, \ldots\}$, $\mathbb{N}_{0} := \mathbb{N} \cup 0$, $\mathbb{N}_{\infty} := \mathbb{N} \cup \infty$. We consider \emph{individuals} as being labelled by elements of the infinite \emph{Ulam-Harris} tree $\mathcal{U}_\infty : = \bigcup_{n \geq 0} \mathbb{N}^{n}$; where $\mathbb{N}^{0} := \{\varnothing\}$ contains a single element $\varnothing$ which we call  the \emph{root}. We denote elements $u \in \mathcal{U}_{\infty}$ as a tuple, so that, if $u = (u_{1}, \ldots, u_{k}) \in \mathbb{N}^{k}, k \geq 1$, we write $u = u_{1} \cdots u_{k}$. An individual $u = u_1u_2\cdots u_k$ is to be interpreted recursively as the $u_k$th \emph{child} of the individual $u_1 \cdots u_{k-1}$; for example, $1, 2, \ldots$ represent the offspring of $\varnothing$. We consider a Crump-Mode-Jagers branching process as a random labelling\footnote{Note that the notation we introduce here is actually more general than the classical CMJ branching process, assumed to satisfy Assumption~\ref{ass:classical-cmj}. In principle, one may consider structures with more correlations, or also consider labelling structures that allow these labels to take values in an arbitrary vector space instead of $[0,\infty]$.} of the elements of $\mathcal{U}_{\infty}$ with elements of $[0, \infty]$, one considers \emph{birth-times}.
In order to do so, we assume the existence of a complete probability space $(\Omega, \Sigma, \mathbb{P})$ and equip $\mathcal{U}_{\infty}$ with the sigma algebra generated by singleton sets. We then, assume the existence of a random (measurable) function $X: \Omega \times \mathcal{U}_{\infty} \rightarrow [0, \infty]$, where, for $u \in \mathcal{U}_{\infty}$, we think of $X(uj)$ as the \emph{displacement} or \emph{waiting time} between the $(j-1)$th and $j$th child of $u$. We use the values of $X$ to associate \emph{birth times} $\mathcal{B}(u)$ to individuals $u \in \mathcal{U}_{\infty}$. In particular, we define $\mathcal{B}: \Omega \times \mathcal{U}_{\infty} \rightarrow [0, \infty]$ recursively as follows:  
\[\mathcal{B}(\varnothing) : = 0 \quad \text{and for $u \in \mathcal{U}_{\infty}, i \in \mathbb{N}$,} \quad \mathcal{B}(ui) := \mathcal{B}(u) + \sum_{j=1}^{i} X(uj).\]
Consequentially, a value of $X(ui) = \infty$ indicates that the individual $u$ has stopped producing offspring, and does not produce $i$ children or more. 

We introduce some notation related to elements $u \in \mathcal{U}_{\infty}$: we use $|\cdot|$ to measure the \emph{length} of a tuple $u$, so that, if $u = \varnothing$ we set $|u| = 0$, whilst if $u = u_{1} \cdots u_{k}$ then $|u| = k$. %For $n \in \mathbb{N}_{0}$, we set $\mathcal{U}_{n} : = \left\{u \in \mathcal{U}_{\infty}: |u| = n\right\}$ (really, $\mathcal{U}_{n}$ is the same as $\mathbb{N}^{n}$, but we feel the prior notation is clearer). 
%If, for some $x \in \mathcal{U}_{\infty}$, we have $x = u v$, we say $u$ is a \emph{ancestor} of $x$. We introduce a notation to refer to ancestors: 
Given $\ell \leq |u|$, we set $u_{|_\ell} := u_{1} \cdots u_{\ell}$.  
%It will be helpful to equip $\mathcal{U}_{\infty}$ with the lexicographic total order $\leq_{L}$: given elements $u, v$ we say $u \leq_{L} v$ if either $u$ is a ancestor of $v$, or $u_{\ell} < v_{\ell}$ where $\ell = \min \left\{i \in \mathbb{N}: u_{i} \neq v_{i} \right\}$. 
We say a subset $T \subset \mathcal{U}_{\infty}$ is a \emph{tree} if, given that $u \in T$, we also have $u_{|_{\ell}} \in T$, for each $\ell \leq |u|$. Note that any such trees can be viewed as graphs in the natural way, connecting nodes to their children. 

For each $t \in [0, \infty]$, we set $\mathscr{T}_{t} = \{(x, \mathcal{B}(x)) \in \mathcal{U}_{\infty} \times [0, \infty]: \mathcal{B}(x) \leq t\}$, we think of this as the \emph{genealogical tree} of individuals with birth time at most $t$. We also define for each $k \in \mathbb{N}_{0}$, $\gen{k} : = \left\{(x, \mathcal{B}(x)) \in \mathcal{U}_{\infty} \times [0, \infty]: |x| \leq k \right\}$. We set $\mathscr{T}_{\infty} := \bigcup_{k \in \mathbb{N}} \gen{k}$. 
We denote by $(\mathcal{F}_{t})_{t \geq 0}$ and $(\mathcal{G}_{k})_{k \in \mathbb{N}_{0}}$ the natural filtrations generated by $(\mathscr{T}_{t})_{t \geq 0}$ and $(\gen{k})_{k \in \mathbb{N}}$ respectively\footnote{More formally, we view $(\mathscr{T}_{t})_{t \geq 0}$ as a measurable mapping 
$\Omega \times [0, \infty) \times \mathcal{U}_{\infty} \rightarrow [0, \infty]$, where, $\mathscr{T}_{t}(u) = \mathcal{B}(u) \mathbf{1}_{\mathcal{B}(u) \leq t} + \infty \mathbf{1}_{\mathcal{B}(u) > t}$. Likewise, we view $(\gen{k})_{k \in \mathbb{N}_{0}}$ as a measurable map $\Omega \times \mathbb{N}_{0} \times \mathcal{U}_{\infty} \rightarrow [0, \infty]$, where, $\gen{k}(u) = \mathcal{B}(u) \mathbf{1}_{|u| \leq k}$. Also, as a formality, by taking completions if necessary, we assume that each $\mathcal{F}_{t}$ and $\mathcal{G}_{k}$ is complete.}. In relation to the process $(\mathscr{T}_{t})_{t \geq 0}$, we define the stopping times $(\tau_{k})_{k \in \mathbb{N}_{0}}$ such that $\tau_{k} := \inf\{t \geq 0: |\mathscr{T}_{t}| \geq k\}$. One readily verifies that $(|\mathscr{T}_{t}|)_{t \geq 0}$ is right-continuous, and thus $|\mathscr{T}_{\tau_{k}}| \geq k$. %For each $k \in \mathbb{N}$ we define the tree $\mathcal{T}_{k}$ as the tree consisting of the first $k$ individuals in $\mathscr{T}_{\tau_{k}}$ ordered by birth time, breaking ties lexicographically.
%One readily verifies that $(|\mathscr{T}_{t}|)_{t \geq 0}$ is right-continuous, and thus $|\mathscr{T}_{\tau_{k}}| \geq k$. For each $k \in \mathbb{N}$ we define the tree $\mathcal{T}_{k}$ as the tree consisting of the first $k$ individuals in $\mathscr{T}_{\tau_{k}}$ ordered by birth time, breaking ties lexicographically. 
We call $\tau_{\infty} := \lim_{k \to \infty} \tau_{k}$ the \emph{explosion time} of the process. For a given choice of $X$, we say that $\mathscr{T}_{\infty}$ is an $X$-\emph{Crump-Mode-Jagers (CMJ) branching process}. We also use this term in regards to the stochastic processes $(\mathscr{T}_t)_{t \geq 0}$ and $(\gen{k})_{k \in \mathbb{N}_{0}}$. In this paper we denote by $\mathcal{S}$ the event
\begin{equation} \label{eq:survival}
    \mathcal{S} := \left\{\forall k \in \mathbb{N} \, \exists u \in \mathbb{N}^{k}: \mathcal{B}(u) < \infty\right\}. 
\end{equation}
In words, this event indicates that individuals $u$ of arbitrarily large size $|u|$ are born. We call this event \emph{survival}, and the complementary event $\mathcal{S}^{c}$ \emph{extinction}. 
%We say that the process $\mathscr{T}_{\infty}$ \emph{survives} if for all $k \in \mathbb{N}$ there exists $u \in \mathbb{N}^{k}$ such that $\mathcal{B}(u) < \infty$; in other words,  In this paper, we denote the event of survival by $\mathcal{S}$. 

It will be helpful to have notation concerning ``shifts'' of $\mathscr{T}_{\infty}$ representing the sub-tree associated with some $u \in \mathcal{U}_{\infty}$. With regards to birth-times, for any $v \in \mathcal{U}_{\infty}$ we define 
\[
\mathcal{B}^{(u)}(v) := 
\begin{cases}
\mathcal{B}(uv) - \mathcal{B}(u) & \text{if $\mathcal{B}(u) < \infty$}
\\ \infty & \text{otherwise.}
\end{cases}
\]
For $u \in \mathcal{U}_{\infty}$, we define $\mathscr{T}_{\infty}^{\downarrow u} : = \left\{(x, \mathcal{B}^{(u)}(x)), x \in \mathcal{U}_{\infty} \right\}$.

For each $u \in \mathcal{U}_{\infty}$ it will be helpful to a have a map $\xi^{(u)}: \Omega \times [0, \infty] \rightarrow \mathbb{N}_{0}$ indicating the number of children $u$ has produced, more precisely, we define $\xi^{(u)}(t)$ such that 
\[
\xi^{(u)}(t) = \begin{cases}
	\sum_{i=1}^{\infty}\mathbf{1}_{\{ \mathcal{B}^{(u)}(i)\leq \mathcal{B}(u) + t\}} &  \text{if $\mathcal{B}(u) < \infty$;}
	\\ 0 & \text{otherwise.}
\end{cases}
\] 
We denote by $\xi(t) := \xi^{(\varnothing)}(t)$. %In the above framework, we will consider general features of the process as being determined by $\mathscr{T}_{\infty}$\footnote{More formally, the completion of the sigma algebra generated by $\mathscr{T}_{\infty}$.}. 

%Given random variable $Z$ taking values in $\mathbb{N}_{\infty}$, we denote by $G_{Z}(x)$ the probability generating function associated with $Z$. Note then, that by classical theory associated with Bienaym\'{e}-Galton-Watson branching processes, the value $\Prob{\mathcal{S}^{c}} = 1 - \Prob{\mathcal{S}}$ is given by the minimal solution of the equation
%\[G_{\xi(\infty)}(x) = x, \quad x \in [0, 1]. \]
%Note also, that since $G_{\xi(\infty)}(x) \leq x$, such a solution is also the minimum solution of the equation $x \leq G_{\xi(\infty)}(x)$, $x \in [0, 1]$. 

\subsection{Measure theoretical technicalities} \label{sec:measurability}
Formally, we consider $\mathscr{T}_{\infty}$, as a random mapping $\Omega \times \mathcal{U}_{\infty} \rightarrow [0, \infty]$, such that $u \mapsto \mathcal{B}(u)$. Viewing the branching process as a random function from $\mathcal{U}_{\infty}$ to $[0, \infty]$, it is helpful to define a sigma algebra on the space of functions $[0, \infty]^{\mathcal{U}_{\infty}}$, representing \emph{properties} of the branching process that can be ``measured''. In particular, we want to ensure events involving uncountable unions, such as ``there exists an infinite path in finite time'', can be measured. This is the reason for the technicalities in this section.  

In this regard, given $X$, and hence $\mathcal{B}$, we define the function $\iota_{\mathcal{B}}: \Omega \rightarrow [0, \infty]^{\mathcal{U}_{\infty}}$ such that $\omega \mapsto \mathcal{B}(\cdot, \omega)$. We equip $[0, \infty]^{\mathcal{U}_{\infty}}$ with the pushforward sigma algebra induced by the map $\iota_{\mathcal{B}}$, that is, the sigma algebra $\mathcal{J}:= \left\{A \subseteq [0, \infty]^{\mathcal{U}_{\infty}}: \iota_{\mathcal{B}}^{-1}(A) \in \Sigma \right\}$. The following is a foundational definition from descriptive set theory. Here, we reformulate Definition~1.10.1.~\cite{bogachev}:

\begin{definition}[Souslin scheme]
    Given a collection of subsets $\mathcal{C}$ of a given set $S$, a \emph{Souslin scheme} over $\mathcal{C}$ is a mapping that associates to every sequence $n_1 \cdots n_{k}$ of natural numbers, a set $C_{n_1, \cdots n_{k}} \in \mathcal{C}$. A Souslin operation is the map that associates, to every Souslin scheme $\left\{C_{n_1, \cdots n_{k}}, n_1 \cdots n_k \in \mathcal{U}_{\infty} \right\}$ over $\mathcal{C}$, the set 
    \[
    \bigcup_{(n_{i}) \in \mathbb{N}^{\infty}} \bigcap_{k=1}^{\infty} C_{n_1 \cdots n_{k}}. 
    \]
    The collection of such sets, along with the empty set is denoted $\mathcal{A}(\mathcal{C})$.
\end{definition}

The following beautiful theorem is also classical from descriptive set theory, here we reformulate Theorem~1.10.5.~\cite{bogachev}:

\begin{theorem} \label{theorem:closure-souslin}
    Suppose that a sigma algebra $\mathcal{C}$ is complete with respect to a finite, non-negative measure $\mu$. Then $\mathcal{A}(\mathcal{C}) = \mathcal{C}$. 
\end{theorem}

This theorem shows that, since $(\Omega, \Sigma, \mathbb{P})$ is a complete probability space, a number of events, expressed as uncountable unions, are measurable. In particular, if, for $u \in \mathcal{U}_{\infty}$ we define the random variable
\[
\tau_{\tpath}(u) := \inf_{t \geq 0}\left\{\exists (n_k) \in \mathbb{N}^{\infty}: \forall k \in \mathbb{N} \; \mathcal{B}^{(u)}(n_1 \cdots n_{k}) \leq t\right\},
\]
i.e., the first time $t$ such that there exists an infinite path (or \emph{ray}) starting at $u$ with all individuals on the path born before $t$.
Then, as an application of Theorem~\ref{theorem:closure-souslin}, we know that, for any $t \in [0, \infty]$, the event $\left\{\tau_{\tpath}(u) < t\right\}$ is measurable. In particular, that $\tau_{\tpath}(u)$ is a random variable. For brevity, we write $\tau_{\tpath}:= \tau_{\tpath}(\varnothing)$. We say that \emph{vertical explosion} occurs if $\tau_{\tpath} < \infty$. This terminology complements the notion of \emph{sideways explosion} from~\cite{Komjathy2016ExplosiveCB}, when $\sum_{i=1}^{\infty} X(i) < \infty$ with positive probability. %This terminology complements the notion of . 

\section{Main results and proofs} \label{sec:main}

In general in this paper, we will assume the following throughout:
\begin{equation} \label{eq:cmj-assumption}
	((X(uj))_{j \in \mathbb{N}}) \quad \text{are independent \ for different $u \in \mathcal{U}_{\infty}$}.
\end{equation}
Note that this implies that for distinct $u_1, u_2 \in \mathcal{U}_{\infty} $ with $|u_1| = |u_2|$, the values of $(\mathcal{B}^{(u_1)}(v))_{v \in \mathcal{U}_{\infty}}$ and $(\mathcal{B}^{(u_2)}(v))_{v \in \mathcal{U}_{\infty}}$ are independent on $\mathcal{B}(u_1), \mathcal{B}(u_2) < \infty$, hence so are the sub-trees $\mathscr{T}^{\downarrow u_1}_{\infty}, \mathscr{T}^{\downarrow u_2}_{\infty}$. Note that, since $|u_1| = |u_2|$ and are distinct, the sub-trees $\mathscr{T}^{\downarrow u_1}_{\infty}, \mathscr{T}^{\downarrow u_2}_{\infty}$ are disjoint. 

We will also, mostly, require the following assumption:

\begin{assumption} \label{ass:classical-cmj}
    In addition to~\eqref{eq:cmj-assumption}, we have 
    \begin{equation} \label{eq:sub-tree-identical-law}
    \mathscr{T}^{\downarrow u}_{\infty} \sim \mathscr{T}_{\infty} \quad \text{on} \quad \left\{\mathcal{B}(u) < \infty \right\}. 
\end{equation}
More formally, the regular conditional distribution of $\mathscr{T}^{\downarrow u}_{\infty}$ on $\left\{\mathcal{B}(u) < \infty \right\}$ coincides with the distribution of $\mathscr{T}_{\infty}$.
\end{assumption}

Note that Assumption~\ref{ass:classical-cmj} is the classical definition of a CMJ branching process, and is the same as assuming that $((X(uj))_{j \in \mathbb{N}})$ are i.i.d. \ for different $u \in \mathcal{U}_{\infty}$, hence this assumption is stronger than~\eqref{eq:cmj-assumption}.
%\begin{assumption} \label{eq:cmj-gen-dependence}
%    For all $u_{1}, u_{2} \in \mathcal{U}_{\infty}$ with $|u_1| = |u_2|$, we
%    \begin{equation} \label{eq:sub-tree-different-law}
%    \mathscr{T}^{\downarrow u}_{\infty} \sim \mathscr{T}_{\infty} \quad \text{on} \quad \left\{\mathcal{B}(u) < \infty \right\},
%    \end{equation}
%\end{assumption}
%As a result of this independence across generations, we have 

%it is straightforward to verify that $\left| \left\{u \in \mathcal{U}_{\infty} : \mathcal{B}(u) < \infty\right\} \right|$ is the size of the genealogical tree associated with a Bienaym\'{e}-Galton-Watson branching process with offspring distribution defined according to $\xi(\infty)$ (thus also allowing the possibility for a single individual to produce infinitely many children). 

\subsection{Inherited properties and a conditional $0-1$ law} \label{sec:inherited}

Given a set $\mathcal{P} \in \mathcal{J}$, we say that $\mathscr{T}_{\infty}$ has \emph{property} $\mathcal{P}$ if $\mathscr{T}_{\infty} \in \mathcal{P}$. Using the map $\iota_{\mathcal{B}}: \Omega \rightarrow [0,\infty]^{\mathcal{U}_{\infty}}$, we often identify properties with events $B \in \Sigma$ in the probability space $(\Omega, \Sigma, \mathbb{P})$. Recall the event $\mathcal{S}$ from~\eqref{eq:survival}. 

\begin{definition}[Inherited and bidirectionally inherited properties]
     We say that a property $\mathcal{P}$ is \emph{inherited} if, whenever $\mathscr{T}_{\infty} \in \mathcal{P}$, then, for each $i$ with $\mathcal{B}(i) < \infty$, we have $\mathscr{T}_{\infty}^{\downarrow i} \in \mathcal{P}$, and, moreover, $\mathscr{T}_{\infty} \in \mathcal{P}$ if the event $\mathcal{S}^{c}$ occurs. We say that a property $\mathcal{P}$ is \emph{bidirectionally 
     inherited} if it is inherited and it is the case that $\mathscr{T}_{\infty} \in \mathcal{P}$, if and only if for each $i$ with $\mathcal{B}(i) < \infty$, we have $\mathscr{T}_{\infty}^{\downarrow i} \in \mathcal{P}$. 
\end{definition}

\begin{remark} \label{rem:inherited-survival}
    Since the event of extinction $\mathcal{S}^{c}$ implies that extinction also occurs in the tree $\mathscr{T}_{\infty}^{\downarrow i}$, it follows that any inherited property occurs on extinction. 
\end{remark}

\begin{example}
    Some examples of inherited properties include $\{\tau_{\tpath} < \infty\}^{c}$ (i.e., the event "there is no infinite path in finite time") and $\{\tau_{\infty} < \infty\}^{c}$. %Indeed, if the process does not produce an infinite path in finite time, there cannot be a sub-tree of a child of the root born in finite time that produces an infinite path in finite time. Similar logic applies to the event of explosion. 
    The event $\{\tau_{\tpath} < \infty\}^{c}$ is also bidirectionally inherited, however, one may construct counter examples to show that $\{\tau_{\infty} < \infty\}^{c}$ is not bidirectionally inherited in general, when Assumption~\ref{ass:classical-cmj} is not satisfied.
    
    %: if every sub-tree of a child of the root born in finite time does not produce an infinite path in finite time, the process does not produce an infinite path in finite time. However, note that it is not the case that  Indeed, we may, for example, have that $\sum_{j=1}^{\infty} X(j) < \infty$ almost surely, whilst for each $u \in \mathcal{U}_{\infty}$ with $|u| \geq 1$, $X(uj) > \eps$ almost surely. Then, every sub-tree of a child of the root born in finite time may be non-explosive, but the process still explodes.  
\end{example}

The following is well-known for inherited properties, at least with respect to Bienaym\'{e}-Galton-Watson processes %and CMJ processes satisfying Assumption~\ref{ass:classical-cmj} 
(see, e.g.,~\cite[Proposition~5.6]{lyons}). We provide proof here that works in a rather general setting. 

\begin{proposition}[Conditional $0-1$ law] \label{prop:cond-zero-one}
    Let $\mathscr{T}_{\infty}$ be an $X$-CMJ branching process. Then, we have the following:
    \begin{enumerate}
        \item \label{item:class-cmj} Suppose Assumption~\ref{ass:classical-cmj} is satisfied. Then, for any inherited property $\mathcal{P}$, either $\Prob{\mathscr{T}_{\infty} \in \mathcal{P}} =1$ or $\Prob{\left\{\mathscr{T}_{\infty} \in \mathcal{P}\right\} \cap \mathcal{S}} = 0$. In particular, if we have $\Prob{\mathcal{S}} > 0$, then $\Prob{\mathscr{T}_{\infty} \in \mathcal{P} \, | \, \mathcal{S}} \in \{0, 1\}$.
        \item \label{item:det-prop} More generally, suppose that only ~\eqref{eq:cmj-assumption} is satisfied, and in addition, for each $k \in \mathbb{N}$, the set $\left\{u \in \mathbb{N}^{k}: \mathcal{B}(u) < \infty \right\}$ is deterministic. Then, for any bidirectionally inherited property $\mathcal{P}$, we have $\Prob{\mathscr{T}_{\infty} \in \mathcal{P}} \in \left\{0, 1\right\}$.
    \end{enumerate}
\end{proposition}

\begin{remark}
    Note that, in Item~\ref{item:det-prop}, the assumption that $\left\{u \in \mathbb{N}^{k}: \mathcal{B}(u) < \infty \right\}$ is deterministic implies that $\mathcal{S}$ is a deterministic property.
\end{remark}

\begin{proof}[Proof of Proposition~\ref{prop:cond-zero-one}]
First recall from Remark~\ref{rem:inherited-survival} that for any inherited property $\mathcal{P}$, $\mathcal{S}^{c}$ implies $\mathscr{T}_{\infty} \in \mathcal{P}$. Also, recall the definition of $\mathcal{G}_{k}$, the sigma algebra generated by the birth times of all individuals up to the $k$th generation of the process. 
Now, $\mathscr{T}_{\infty} \in \mathcal{P}$ implies that, for all $i \in \mathbb{N}$, with $\mathcal{B}(i) < \infty$ we have $\mathscr{T}^{\downarrow i}_{\infty} \in \mathcal{P}$. Applying this recursively, for all $x \in \mathbb{N}^{k}$ with $\mathcal{B}(x) < \infty$, we have $\mathscr{T}^{\downarrow x}_{\infty} \in \mathcal{P}$. By~\eqref{eq:cmj-assumption} the sub-trees $(\mathscr{T}^{\downarrow u}_{\infty}: |u| = k, \mathcal{B}(u) < \infty)$ are mutually independent, conditional on the values $\mathcal{B}(u), |u| = k$. Thus, 
\begin{linenomath}
\begin{align} \label{eq:lev-martingale}
\nonumber \Prob{\mathscr{T}_{\infty} \in \mathcal{P} \, | \, \mathcal{G}_{k}} & \leq \left(\prod_{u \in \mathbb{N}^{k}: \mathcal{B}(u_i) < \infty} \Prob{ \mathscr{T}^{\downarrow u}_{\infty} \in \mathcal{P}\, \bigg | \,  \mathcal{B}(u) < \infty} \right) \mathbf{1} \left\{\exists u \in \mathbb{N}^{k} : \mathcal{B}(u) < \infty\right\} \\ & \hspace{5cm}  + \mathbf{1} \left\{\forall u \in \mathbb{N}^{k} : \mathcal{B}(u) = \infty\right\}
\\ & \leq  \Prob{\mathscr{T}_{\infty} \in \mathcal{P}} \mathbf{1} \left\{\exists u \in \mathbb{N}^{k} : \mathcal{B}(u) < \infty\right\} + \mathbf{1} \left\{\forall u \in \mathbb{N}^{k} : \mathcal{B}(u) = \infty\right\},
\end{align}
\end{linenomath}
where the second inequality follows from the fact that, by~\eqref{eq:sub-tree-identical-law}, on the event $\left\{\mathcal{B}(u) < \infty\right\}$, the tree $\mathscr{T}^{\downarrow u}_{\infty}$ has the same law as $\mathscr{T}_{\infty}$. Taking limits as $k \to \infty$, by the L\'{e}vy zero-one law, the left-hand side is an indicator random variable, whilst the two indicators on the right-hand side converge almost surely to $\mathbf{1}\left\{\mathcal{S} \right\}$ and $\mathbf{1}\left\{\mathcal{S}^{c}\right\}$ respectively. Thus, almost surely, we have
\begin{linenomath}
\begin{align} \label{eq:zero-one-positive-set}
\mathbf{1} \left\{ \mathscr{T}_{\infty} \in \mathcal{P}\right\}  & \leq \Prob{ \mathscr{T}_{\infty} \in \mathcal{P}} \mathbf{1}\left\{\mathcal{S} \right\} + \mathbf{1}\left\{\mathcal{S}^{c}\right\}.
\end{align}
\end{linenomath}
Therefore, if there exists a non-null set of $\omega \in \Omega$ such that we have both $\mathbf{1} \left\{ \mathscr{T}_{\infty} \in \mathcal{P}\right\}(\omega) = 1$ and $\mathbf{1}\left\{\mathcal{S} \right\}(\omega) =1$, then~\eqref{eq:zero-one-positive-set} implies that $\Prob{\mathscr{T}_{\infty} \in \mathcal{P}} = 1$. But this is equivalent to $\Prob{ \left\{ \mathscr{T}_{\infty} \in \mathcal{P}\right\} \cap \mathcal{S}} > 0$ which implies Item~\ref{item:class-cmj}. 
For Item~\ref{item:det-prop}, if $\mathcal{P}$ is bidirectionally recursive, we have
\begin{equation} \label{eq:kolmogorov}
\Prob{\mathscr{T}_{\infty} \in \mathcal{P}} = \Prob{ \bigcap_{k \in \mathbb{N}} \bigcap_{u \in \mathbb{N}^{k}: \mathcal{B}(u_i) < \infty} \mathscr{T}^{\downarrow u}_{\infty} \in \mathcal{P}}. 
\end{equation}
Suppose $F_{k}$ denotes the sigma algebra generated by $(X(uj)_{j \in \mathbb{N}, u \in \mathbb{N}^{k}})$. The fact that $\left\{u \in \mathbb{N}^{k}: \mathcal{B}(u) < \infty \right\}$ is deterministic implies that the right-hand side of~\eqref{eq:kolmogorov} is a tail event with respect to $(F_{k})_{k \in \mathbb{N}_{0}}$, and by~\eqref{eq:cmj-assumption} we may apply the Kolmogorov $0-1$ law. 
\end{proof}
\begin{remark}
Note that the proof of Item~\ref{item:class-cmj} in Proposition~\ref{prop:cond-zero-one} can be made more general than assuming Assumption~\ref{ass:classical-cmj}. Indeed, for~\eqref{eq:lev-martingale}, we need only assume that, for 
$u1 \in \mathcal{U}_{\infty}$, with $u \in \mathcal{U}_{\infty}$ (i.e, the first child of an individual), we have $\mathscr{T}^{\downarrow u1}_{\infty} \sim \mathscr{T}_{\infty}$ on $\left\{\mathcal{B}(u1) < \infty \right\}$. 
\end{remark}

\subsection{A sufficient criteria for vertical explosion}
\begin{theorem} \label{theorem:vert-expln}
Let $\mathscr{T}_{\infty}$ be an $X$-CMJ branching process satisfying Assumption~\ref{ass:classical-cmj}. Assume that there exists two sequences $(t_{i}) \in [0,\infty)^{\mathbb{N}} $ and $(M_{i}) \in \mathbb{N}^{\mathbb{N}}$ satisfying the following: we have $\sum_{i=1}^{\infty} t_{i} < \infty$ and 
\begin{equation} \label{eq:summ-condition}
\sum_{i=1}^{\infty} \Prob{\xi(t_{i}) \leq M_{i+1}}^{M_{i}} < \infty.
\end{equation} 
Then, $\Prob{\tau_{\tpath} < \infty \,| \, \Supe} = 1$.
\end{theorem}

\begin{remark}
    Note that, since $\Prob{\tau_{\tpath} < \infty \,| \, \Supe} \leq \Prob{\tau_{\infty} < \infty \,| \, \Supe}$, Theorem~\ref{theorem:vert-expln} also provides a sufficient condition for
    $\mathscr{T}_{\infty}$ to be explosive on survival. Note also that in the case that $\xi(t) < \infty$ almost surely, for all $t > 0$, the tree $\mathscr{T}_{\tau_{\infty}}$, must be locally finite. As any locally finite graph has an infinite path by K\H{o}nig's lemma, in this case, $\tau_{\infty}$ and $\tau_{\tpath}$ coincide.   
\end{remark}

\begin{remark} \label{remark:non-contradiction}
    Note that the assumptions of Theorem~\ref{theorem:vert-expln} cannot be satisfied if $\E{\xi(t_0)} < 1$, for some $t_0 > 0$. Indeed, any candidate sequences $(t_{i})_{i \in \mathbb{N}}$, $(M_i)_{i \in \mathbb{N}}$ must contain infinitely many terms such that $M_{i+1}/M_{i} \geq 1$ and $t_{i} \leq t_{0}$, for which  
    \begin{linenomath}
    \begin{align*}
    \Prob{\xi(t_{i}) \leq M_{i+1}}^{M_{i}} & \geq (1- \Prob{\xi(t_{0}) > M_{i+1}})^{M_{i}} \geq \left(1 - \frac{M_{i}}{M_{i+1}}\E{\xi(t_{0})}\right) > 0, 
    \end{align*}
    \end{linenomath}
    where the second to last inequality uses Markov's and Bernoulli's inequalities. This implies that the sum in~\eqref{eq:summ-condition} diverges. In fact, if $\E{\xi(t_0)} < 1$, for some $t_0 > 0$ one can adapt the proof of~\cite[Theorem~3.1(b)]{Komjathy2016ExplosiveCB}, to show that actually $\tau_{\tpath} = \infty$ almost surely. 
\end{remark}

Before we prove Theorem~\ref{theorem:vert-expln}, we prove the following corollary:
\begin{corollary} \label{cor:vert-expln}
Let $\mathscr{T}_{\infty}$ be an $X$-CMJ branching process satisfying Assumption~\ref{ass:classical-cmj}.
Suppose that there exists $\eps, \eps', x_0 > 0$, such that, for all $t \in (0, \eps')$, and $x \geq x_0$, we have 
\[
\Prob{\xi(t) > x} > x^{-1} (\log{x})^{1+ \eps} t.  
\]
Then $\Prob{\tau_{\tpath} < \infty \,| \, \Supe} = 1$.
\end{corollary}

\begin{proof}[Proof of Corollary~\ref{cor:vert-expln}]
    In Theorem~\ref{theorem:vert-expln}, we choose $M_{i} := 2^{i}$, and $t_{i} := i^{-(1+ \eps/2)}$. Then, note that
    \[\Prob{\xi(t_i) > 2^{i+1}} > 2^{-(i+1)}(\log{2})^{1+\eps} (i+1)^{1+\eps} i^{-(1+ \eps/2)} > 2^{-(i+1)}(\log{2})^{1+\eps} i^{\eps/2}.\]
    Therefore, we have
    \begin{linenomath}
        \begin{align*}
            \Prob{\xi(t_{i}) \leq 2^{i+1}}^{2^{i}} = \left(1 - \Prob{\xi(t_i) > 2^{i+1}}\right)^{2^{i}} \leq e^{-2^{i}\Prob{\xi(t_i) > 2^{i+1}}} \leq e^{-i^{\eps/2} \frac{(\log{2})^{1+\eps}}{2}}. 
        \end{align*}
    \end{linenomath}
    The above bound is summable in $i$, hence we conclude the result. 
\end{proof}

\begin{proof}[Proof of Theorem~\ref{theorem:vert-expln}]
In order to prove this theorem, we first define a nested sequence of events $\mathcal{E}_{1} \supseteq \mathcal{E}_{2} \cdots$.   Suppose that $[M_{i}] := \left\{1, \ldots, M_{i} \right\}$. Then, for $i \in \mathbb{N}$ we define  
\[
\mathcal{E}_{i} := \left\{\exists u_1 \in [M_1], \ldots, \exists u_{i} \in [M_i]: \xi^{(u_1)}(t_1) > M_2, \ldots, \xi^{(u_1\cdots u_{i})}(t_{i}) > M_{i+1}\right\}. 
\]
Suppose $t^{*} = \sum_{i=1}^{\infty} t_{i}$. The event $\bigcap_{i=1}^{\infty}\mathcal{E}_{i}$ guarantees that
\[|\left\{u = u_1\cdots u_{k} \in \mathcal{U}: u_1 \leq M_1, \ldots, u_{k} \leq M_{k}, \mathcal{B}(u) \leq t^{*}\right\}| = \infty, \]
and as the elements of this set form a locally finite tree sub-tree of $\mathcal{U}$, K\H{o}nig's lemma implies there exists an infinite path in this tree. In particular,  $\bigcap_{i=1}^{\infty} \mathcal{E}_{i} \subseteq \left\{\tau_{\tpath} < \infty\right\}$.
Now, suppose $\mathcal{E}_{i}$ is satisfied by the existence of a path $u_1, \ldots, u_{i}$. Then, for $\mathcal{E}_{i+1}$ to be satisfied one need only find a node $u_{i+1}$ extending this path such that $\xi^{(u_1, \ldots, u_{i}, u_{i+1})}(M_{i+2}) \leq t_{i+1}$. In addition, conditional on $u_1 \cdots u_{i}$, the processes $\xi^{(u_1\cdots u_{i} 1)}, \ldots, \xi^{(u_1\cdots u_{i} M_{i+1})}$ are independent of $u_1 \cdots u_{i}$, and have the same distribution as $\xi^{(1)}, \ldots, \xi^{(M_{i+1})}$. 
Thus, 
\begin{linenomath}
\begin{align*}
    \Prob{\mathcal{E}_{i+1} \, \bigg | \, \mathcal{E}_{i}} & \geq  \Prob{\exists u_{i+1} \in [M_{i+1}]: \xi^{(u_{i+1})}(t_{i+1}) > M_{i+2}} \\ & = 1- \prod_{u_{i+1} =1}^{M_{i+1}} \Prob{\xi^{(u_{i+1})}(t_{i+1}) \leq M_{i+2}} = 1 - \Prob{\xi(t_{i+1}) \leq M_{i+2}}^{M_{i+1}}.
\end{align*}
\end{linenomath}
Now, since $\sum_{\ell=1}^{\infty} \Prob{\xi(t_{\ell}) \leq M_{\ell+1}}^{M_{\ell}} < \infty$, we have
\[
\Prob{\bigcap_{i=1}^{\infty} \mathcal{E}_{i}} = \lim_{j \to \infty} \Prob{\mathcal{E}_{j}} = \lim_{j \to \infty} \prod_{\ell = 1}^{j} \Prob{\mathcal{E}_{\ell} \, \bigg | \, \mathcal{E}_{\ell -1}}  \geq \prod_{\ell=1}^{\infty} \left(1 - \Prob{\xi(t_{\ell}) \leq M_{\ell+1}}^{M_{\ell}}\right) > 0.
\]
 As $\{\tau_{\tpath} < \infty\}$ is an inherited property, we deduce the result by applying Proposition~\ref{prop:cond-zero-one}. 
\end{proof}

\begin{remark} \label{remark:weakening}
It is possible to weaken Assumption~\ref{ass:classical-cmj} in the proof of Theorem~\ref{theorem:vert-expln}, to include some dependence of $\xi^{(u)}$ on $|u|$. For example, we could assume that each $\xi^{(u)}$ with $|u| = i$ were i.i.d. on $\mathcal{B}(u) < \infty$, depending only on $i$. If $\xi_{i}$ denotes the distribution of $\xi^{(u)}$, on $\mathcal{B}(u) < \infty$, with $|u| = i$, then the $i$th summand in~\eqref{eq:summ-condition} could be replaced by $\Prob{\xi_{i}(t_{i}) \leq M_{i+1}}^{M_{i}}$. This would, in this case, provide criteria for $\Prob{\tau_{\tpath} < \infty} > 0$, and, if we know that for each $i \in \mathbb{N}_{0}$ we have $\lim_{t \to \infty} \xi_{i}(t) = \infty$, Item~\ref{item:det-prop} of Proposition~\ref{prop:cond-zero-one} implies that $\Prob{\tau_{\tpath} < \infty} = 1$. 
\end{remark}

\subsection{Applications to recursive trees with fitness} \label{sec:applications}

In this section, we apply our results to a discrete model of \emph{recursive trees with fitness}. We consider these trees as being rooted with edges directed away from the root. Given a vertex labelled $v$ in a directed tree $T$ we denote its out-degree by $\outdeg{v, T}$.  

\begin{definition}[Recursive tree with fitness]\label{def:wrt}
	Suppose that $(W_{i})_{i \in \mathbb{N}}$ are i.i.d.\ copies of a random variable $W$ that takes values in $\mathcal{W}$, and let $f:\mathbb{N}_{0}\times S \rightarrow [0, \infty)$ denote a fitness function. A $(W, f)$-recursive tree with fitness is the sequence of random trees $(\mathcal{T}_{i})_{i \in \mathbb{N}}$ such that: $\mathcal{T}_{0}$ consists of a single node $0$ with weight $W_{0}$ and for $n \geq 1$, $\mathcal{T}_{n}$ is updated recursively from $\mathcal{T}_{n-1}$ by:
	\begin{enumerate}
		\item Sampling a vertex $j \in \mathcal{T}_{n-1}$ with probability proportional to its fitness, i.e., with probability 
		\begin{equation} \label{eq:rec-step}
			\frac{f(\outdeg{j, \mathcal{T}_{n-1}}, W_j)}{\mathcal{Z}_{n}} \quad \text{ with } \mathcal{Z}_{n} : = \sum_{j=0}^{n-1} f(\outdeg{j, \mathcal{T}_{n-1}}, W_j),
		\end{equation}
  (terminating the process if $\mathcal{Z}_{n} = 0$). 
		\item Connect $j$ with an edge directed outwards to a new vertex $n$ with weight $W_{n}$.
	\end{enumerate}
\end{definition}

Note that, by the property of taking minima of independent exponential random variables, and the memory-less property, this process of trees is the discrete time skeleton process of the following continuous time Markov chain $(\mathscr{T}_{t})_{t \geq 0}$. Nodes $i \in \mathbb{N}_{0}$ have associated i.i.d random weights $W_i \in \mathcal{W}$, and associated independent sequences $(X^{(i)}(j))_{j \in \mathbb{N}}$ with $X^{(i)}(j)$ exponentially distributed with parameter $f(j-1, W_{i})$. The tree $\mathscr{T}_{0}$ consists of an initial node $0$ with random weight $W_0$ and associated exponential random variable $X^{(0)}(1)$. Then, recursively, at the time $\tau_{n}$ corresponding to the elapsure of the $n$th exponential random variable in the process, $X^{(j)}(\outdeg{j, \mathscr{T}_{\tau_{n-1}}} + 1)$, say, 
\begin{itemize}
    \item Connect $j$ with a new vertex $n$ with weight $W_{n}$, and initiate new exponential random variables $X^{(n)}(1)$ and $X^{(j)}(\outdeg{j, \mathscr{T}_{\tau_{n}}} + 1)= X^{(j)}(\outdeg{j, \mathscr{T}_{\tau_{n-1}}} + 2).$
\end{itemize}
As we may consider $n$ as the \emph{child} of $j$, after re-labelling, this continuous time Markov process is an $X$-CMJ branching process satisfying Assumption~\ref{ass:classical-cmj}, where the values of $(X(i), i \in \mathbb{N})$ are such that, conditional on a random weight $W \in \mathcal{W}$, $X(i)$ is an independent exponential random variable with parameter $f(i-1, W)$. We thus have the following proposition, whose proof we omit. Suppose $\stackrel{d.}{=}$ denotes equality in distribution. 

\begin{proposition} \label{prop:equivalence-law-rif}
    If $(\mathscr{T}_{t})_{t \geq 0}$ is an $X$-CMJ branching process as described above, and $(\mathcal{T}_{i})_{i \in \mathbb{N}_{0}}$ is a $(W, f)$-recursive tree with fitness, then, after re-labelling nodes in $(\mathscr{T}_{\tau_{n}})_{n \in \mathbb{N}_0}$ according to order of arrival in the process, we have $(\mathscr{T}_{\tau_{n}})_{n \in \mathbb{N}_0} \stackrel{d.}{=} (\mathcal{T}_{i})_{i \in \mathbb{N}_{0}}$. \hfill $\blacksquare$ 
\end{proposition}

\begin{remark}
    Recall that in the associated $X$-CMJ process we either have $\Prob{\tau_{\infty} = \infty} = 1$ or $\Prob{\tau_{\infty} = \infty} = 1- \Prob{\mathcal{S}}$. By properties of the exponential distribution, conditional on $(\mathcal{Z}_{n})_{n \in \mathbb{N}_0}$, we have $\tau_{\infty} \stackrel{d.}{=} \sum_{n=0}^{\infty} Z_{n}$, with $Z_{n}$ exponentially distributed with parameter $\mathcal{Z}_{n}$. As this is finite if and only if $\sum_{n=0}^{\infty} \mathcal{Z}_{n}^{-1} < \infty$, we deduce that the event $\sum_{n=0}^{\infty} \mathcal{Z}_{n}^{-1} < \infty$ with probability $\Prob{\tau_{\infty} = \infty} = 1$ or $\Prob{\tau_{\infty} = \infty} = 1- \Prob{\mathcal{S}}$. This is not so clear to see by other means: the event $\sum_{n=0}^{\infty} \mathcal{Z}_{n}^{-1} < \infty$ is only a tail event %(at least with respect to the natural filtration) 
    in particular circumstances. 
\end{remark}

\subsubsection{A classification theorem}
Our first result concerns the structure of the infinite tree $\mathcal{T}_{\infty} :=  \bigcup_{i \in \mathbb{N}} \mathcal{T}_{i}$, associated with a $(W, f)$-recursive tree with fitness. Due to Proposition~\ref{prop:equivalence-law-rif}, we identify the trees $(\mathcal{T}_{n})_{n \in \mathbb{N}_{0}}$ with the trees $(\mathscr{T}_{\tau_{n}})_{n \in \mathbb{N}_{0}}$ in the corresponding $X$-CMJ process and continue to refer to notation with regards to explosion in this process. Given a weight $w \in \mathcal{S}$, we denote by \[\de_{\max{}}(w) := \sup{\left\{i : f(i, w) > 0\right\}}.\]  The following theorem applies~\cite[Theorem~2.12]{inhom-sup-pref-attach}. 

\begin{theorem} \label{theorem:rec-trees-structure}
Suppose $\mathcal{T}_{\infty}$ denotes the limiting infinite tree associated with a $(W, f)$-recursive tree with fitness. Then, 
\begin{enumerate}
    \item  \label{item:non-explosion} On the event $\left\{\tau_{\infty} = \infty\right\}$, every vertex $i \in \mathcal{T}_{\infty}$ has $\outdeg{i, \mathcal{T}_{\infty}} =  \de_{\max{}}(w_i) + 1$ (which may be infinite).
    \item \label{item:explosion} On the event $\left\{\tau_{\infty} < \infty\right\}$, there are two cases: 
    \begin{enumerate}
        \item Either the tree $\mathcal{T}_{\infty}$ has a single vertex of infinite degree, and finite height (i.e., the distance of any node from $0$ is finite);
        \item Alternatively, the tree $\mathcal{T}_{\infty}$ has a unique infinite path, and every vertex in $\mathcal{T}_{\infty}$ has finite degree. 
    \end{enumerate}
    \item \label{item:explosion-particular-case} If $\sum_{i=1}^{\infty} \frac{1}{f(i,W)} = \infty$ for almost all $W \in \mathcal{W}$, on the event $\left\{\tau_{\infty} < \infty\right\}$ the tree $\mathcal{T}_{\infty}$ has a unique infinite path, and every vertex in $\mathcal{T}_{\infty}$ has finite degree.  
\end{enumerate}
\end{theorem}
In the proof, we refer to the $X$-CMJ process associated with $\mathcal{T}_{\infty}$, and refer to nodes $i \in \mathcal{T}_{\infty}$ by their Ulam-Harris label. 
\begin{proof}
    For Item~\ref{item:non-explosion}, we first show that each individual $u \in \mathscr{T}_{\infty}$ with $\mathcal{B}(u) < \infty$, and associated weight $W_{u}$ we have $|ui \in \mathcal{U}_{\infty} : \mathcal{B}(ui) < \infty| = \de_{\max{}}(W_u)$. Indeed, one readily checks that $\sum_{i=1}^{\de_{\max{}}(W_{u}) + 1} X(ui) < \infty$ almost surely,
    hence if $\mathcal{B}(u) < \infty$, we deduce the claim. We conclude the result by observing that on the event $\{ \tau_{\infty} = \infty\}$, the tree associated with $\{u \in \mathscr{T}_{\infty}: \mathcal{B}(u) < \infty\}$ coincides with $\mathcal{T}_{\infty}$.

    Item~\ref{item:explosion} is a particular case of~\cite[Theorem~2.12]{inhom-sup-pref-attach}, noting that all of the conditions of~\cite[Assumption~2.11]{inhom-sup-pref-attach} are met. For Item~\ref{item:explosion-particular-case} note that for any $u \in \mathcal{U}_{\infty}$, since $\sum_{i=1}^{\infty} \frac{1}{f(i,W)} = \infty$ we have 
    $\sum_{i=1}^{\infty} X(ui) = \infty$ almost surely.
    But, on $\{\tau_{\infty} < \infty \}$, for there to be a node of infinite degree in $\mathcal{T}_{\infty}$, we require the existence of $u\in \mathcal{U}_{\infty}$ such that  $\sum_{i=1}^{\infty} X(ui) < \infty$, so that all of the children of $u$ are produced in finite time. Since $\mathcal{U}_{\infty}$ is countable, we deduce that this occurs with probability $0$. The statement then follows from Item~\ref{item:explosion}.
\end{proof}

\subsubsection{Applications to linear fitness models, including Bianconi-Barab\'asi trees}
In this section, we apply our results to the $(W, f)$-recursive tree $(\mathcal{T}_{i})_{i \in \mathbb{N}_{0}}$ when $\mathcal{W} := [0, \infty)^2$, and the fitness function takes the form
\begin{equation} \label{eq:linear-fitness}
    f(i, (U, V)) = U i + V.
\end{equation}
This form allows us to apply existing theory regarding pure birth processes with linear rates, to analyse the process $(\xi(t))_{t \geq 0}$ associated with the $X$-CMJ branching process. Indeed, given $u \in \mathcal{U}_{\infty}$  
if $\mathcal{B}(u) < \infty$, conditional on $W_{u} = (U_u,V_u) \in [0, \infty)^2$ the associated process $\xi^{(u)}_{(U_{u}, V_{u})}(t)$ is a pure-birth process, with values in $\mathbb{N}_{0}$, with initial condition $0$, and transitions from $i$ to $i+1$ at rate $U_{u} i + V_{u}$. 

In the statement of Theorem~\ref{theorem:inf-path-trans}, for brevity of notation, we formally set $(e^{xt} -1)/x := t$ when $x = 0$. For example, in~\eqref{eq:linear-fitness} below, this avoids having to make distinctions in the expression based on whether or not $U = 0$. 

\begin{theorem}[Infinite path transition] \label{theorem:inf-path-trans}
    Suppose that $(\mathcal{T}_{i})_{i \in \mathbb{N}_{0}}$
    satisfies~\eqref{eq:linear-fitness}. Assume, moreover, that $\Prob{(U,V) = (0,0)} = 0$. Then, 
    \begin{enumerate}
        \item \label{item:non-expl-linear} If, for some $t > 0$, we have 
            \begin{equation} \label{eq:weight-non-expl-case}
            \E{\frac{V(e^{U t} -1)}{U}} < \infty,
            \end{equation}
            then, every node in $\mathcal{T}_{\infty}$ has infinite degree. 
        \item \label{item:expl-linear} Suppose that the left side of~\eqref{eq:weight-non-expl-case} is infinite for any $t > 0$. Then, if there exists $\eps, \eps', R, x_0 > 0$, such that for all $t \in (0, \eps')$, for all $x \geq x_0$, we have 
            \begin{equation} \label{eq:weight-expl-case}
                \Prob{\frac{V(e^{U t} -1)}{U} > x} >  x^{-1} (\log{x})^{1+ \eps} t,  
            \end{equation}
            then $\mathcal{T}_{\infty}$ is locally finite, with a unique infinite path. 
    \end{enumerate}
\end{theorem}
\begin{example} \label{ex:wrrt-fitness}
    In~\eqref{eq:linear-fitness} the case $U=0$ corresponds to the \emph{weighted random recursive tree}, whilst the case $U = 1$ corresponds to the \emph{preferential attachment model with additive fitness}. Finally, the case $U = V$ corresponds to the \emph{Bianconi-Barab\'asi model}. The following corollary, whose proof we omit, applies to these models. 
\end{example}

\begin{corollary} \label{cor:app}
Consider the examples from Example~\ref{ex:wrrt-fitness}. Then, we have the following:
\begin{enumerate}
    \item \label{item:wrrt} In the weighted random recursive tree and preferential attachment tree with additive fitness, if $\E{V} < \infty$ then every node in $\mathcal{T}_{\infty}$ has infinite degree. When $\E{V} = \infty$, if, for some $\nu, x_0 > 0$, for all $x \geq x_0$ we have $\Prob{V > x} > x^{-1} (\log{x})^{1+ \nu}$, $\mathcal{T}_{\infty}$ is locally finite, with a unique infinite path.
    \item \label{item:bianconi} In the Bianconi-Barab\'asi model, if, for some $t > 0$, we have $\E{e^{Ut}} < \infty$, every node in $\mathcal{T}_{\infty}$ has infinite degree. When $U$ does not admit a moment generating function, if, for $\eps', x_0 >0$, for all $t \in (0, \eps')$, $x \geq x_0$ $\Prob{e^{Ut} > x} > x^{-1} (\log{x})^{1+ \eps} t$, $\mathcal{T}_{\infty}$ is locally finite, with a unique infinite path. 
\end{enumerate}
\hfill $\blacksquare$
\end{corollary}

\begin{remark}
    As far as we know, the results in Corollary~\ref{cor:app} are original. The infinite path regimes in Item~\ref{item:wrrt} encompass the extreme disorder regimes in~\cite{bas, LodOrt20, Lod21}, where the maximal degree associated with the model grows linearly in the size of the tree.
\end{remark}
To prove Theorem~\ref{theorem:inf-path-trans}, we use the following lemma:

\begin{lemma}[Mean and second moment of linear pure birth processes] \label{lemma-appendix-janson}
Let $\left(\mathcal{X}(t)\right)_{t\geq 0}$ be a pure birth process with $\mathcal{X}(0) = 0$ and rates such that $k$ transitions to  $k+1$ at rate $c_1 k + c_2$, for some constants $c_1, c_2 > 0$. Then, for each $t \geq 0$, with $r := c_2/c_1$, we have 
\begin{equation} \label{eq:expected-growth}
\E{\mathcal{X}(t)} = 
\begin{cases}
r (e^{c_1 t} - 1) & \text{if $c_1 > 0$}
\\ c_2 t & \text{if $c_1 = 0$},
\end{cases} \text{ and } 
\end{equation}
whilst 
\begin{equation} \label{eq:expected-growth-second-moment}
\E{(\mathcal{X}(t))^2} =  
\begin{cases}
\frac{r}{r+1}\E{\mathcal{X}(t)}^2 + \E{\mathcal{X}(t)}
& \text{if $c_1 > 0$}
\\ (c_2 t)^2 + c_2 t & \text{if $c_1 = 0$}. 
\end{cases}
\end{equation}
\end{lemma}
\begin{proof}
The cases where $c_1 = 0$ follows from the fact that $\left(\mathcal{X}(t)\right)_{t\geq 0}$ is a homogeneous Poisson process, with rate $c_2$. Otherwise, one may solve the Kolmogorov forward equations associated with the process (see~\cite[Theorem~A.7]{holmgren-janson}) to see that the probability generating function is given by
\begin{equation} \label{eq:pgf}
\E{z^{\mathcal{X}(t)}} = \left(\frac{e^{-c_1t}}{1-z\left(1 - e^{-c_1t}\right)}\right)^{r},
\end{equation}
from which we deduce the claim. 
\end{proof}
As with the statement of Theorem~\ref{theorem:inf-path-trans}, in the proof below, we also formally set $(e^{xt} -1)/x := t$ when $x = 0$. This allows us to unify the different cases appearing in~\eqref{eq:expected-growth} and~\eqref{eq:expected-growth-second-moment} into a single expression. 

\begin{proof}
First note that since $\Prob{(U,V) = (0,0)} = 0$, we have $\Prob{\mathcal{S}} = 1$. 
For Item~\ref{item:non-expl-linear}, note that~\eqref{eq:weight-non-expl-case}, combined with~\eqref{eq:expected-growth} from Lemma~\ref{lemma-appendix-janson}, implies that, for some $t > 0$, with regards to the associated $X$-CMJ branching process, we have $\E{\xi(t)} = \E{\E{\xi_{(U,V)}(t)}}< \infty$. Moreover, note that, $\xi(t) < \infty$ almost surely, for each $t > 0$. By~\cite[Theorem~3.1(b)]{Komjathy2016ExplosiveCB}, we deduce that $\Prob{\tau_{\infty} = \infty} = 1$, which, when combined with Item~\ref{item:non-explosion} of Theorem~\ref{theorem:rec-trees-structure}, implies the claim of Item~\ref{item:non-expl-linear}. 

For Item~\ref{item:expl-linear}, note that on $\left\{\E{\xi_{(U,V)}(t)} > 1\right\}$, we have $\E{\xi_{(U,V)}(t)} \leq \E{\xi_{(U,V)}(t)}^2$, thus by~\eqref{eq:expected-growth-second-moment}
\[
\E{\xi_{(U,V)}(t)^2} = \frac{V/U}{V/U+1} \E{\xi_{(U,V)}(t)}^2 + \E{\xi_{(U,V)}(t)} \leq \frac{2V/U+1}{V/U+1} \E{\xi_{(U,V)}(t)}^2. 
\]
\begin{equation} \label{eq:moment-bound}
\text{Therefore, on $\left\{\E{\xi_{(U,V)}(t)} > 1\right\}$ we have } \frac{\E{\xi_{(U,V)}(t)}^2}{\E{(\xi_{(U,V)}(t))^2}}\geq \frac{V/U+1}{2V/U + 1} > \frac{1}{2}. 
\end{equation}
Thus, applying Paley-Zygmund for the third inequality, for $d \in (0,1)$ for all $x \geq x_0$, $t \in (0, \eps')$ we have 
\begin{linenomath}
\begin{align*} 
\Prob{\xi(t) > x} &= \E{\Prob{\xi_{(U,V)}(t) > x}}  
\\ & \geq  \E{\Prob{\xi_{(U,V)}(t) \geq d \E{\xi_{(U,V)}(t)}} \mathbf{1} \left\{ \E{\xi_{(U,V)}(t)} > x/d \right\}}
\\ & \hspace{-0.03cm} \stackrel{\eqref{eq:moment-bound}}{\geq} \frac{(1-d)^2}{2} \Prob{ \E{\xi_{(U,V)}(t)} > x/d} \stackrel{\eqref{eq:expected-growth}}{=} \frac{(1-d)^2}{2} \Prob{\frac{V(e^{U t} -1)}{U} > x/d}
\\ & \geq \frac{(1-d)^2 (x/d)^{-1} (\log{(x/d)})^{1+ \eps} t}{2}.
\end{align*}
\end{linenomath}
As the constants involving $d$ are negligible compared to $(\log{x})^{\eps}$, we may now choose another $x'_{0} > 0$, and $\eps'' < \eps$ so that for all $x \geq x_{0}', t \in (0, \eps')$ 
\begin{linenomath}
\begin{align} \label{eq:birth-bound}
\Prob{\xi(t) > x} \geq  x^{-1} (\log{x})^{1+ \eps''} t.
\end{align}
\end{linenomath}
By Corollary~\ref{cor:vert-expln}, we deduce that $\tau_{\tpath} < \infty$, almost surely, so that, in particular, $\tau_{\infty} < \infty$ almost surely. As $\sum_{i=1}^{\infty} \frac{1}{Ui + V} = \infty$ almost surely, we may conclude by applying Item~\ref{item:explosion-particular-case} of Theorem~\ref{theorem:rec-trees-structure}.
\end{proof}

\begin{remark}
As Item~\ref{item:expl-linear} of Theorem~\ref{theorem:inf-path-trans} uses Corollary~\ref{cor:vert-expln} in applying~\eqref{eq:birth-bound}, it may be possible to improve this result by applying Theorem~\ref{theorem:vert-expln} directly. 
\end{remark} 

\section*{Acknowledgements}
We thank Bas Lodewijks and Jonas K\"oppl for helpful discussions. The author is funded by Deutsche Forschungsgemeinschaft (DFG) through DFG Project no. $443759178$, and would like to also acknowledge previous funding from the LMS Early Career Fellowship~ECF-1920-55, which helped support this research. 

\bibliographystyle{amsplain}
\bibliography{refs}

\end{document}